\documentclass[12pt]{article}
\usepackage{mathrsfs}
\usepackage{amsmath}
\usepackage{amsthm}
\usepackage{graphicx}
\usepackage{amssymb}

\usepackage{fancyref}

\usepackage{latexsym,bm}
\usepackage{newlfont}
\usepackage{enumerate}
\usepackage{float}
\usepackage{epsfig}
\usepackage{color}
\newtheorem{thm}{Theorem}[section]
\newtheorem{Lemma}[thm]{Lemma}
\newtheorem{cor}[thm]{Corollary}

\makeatletter
\long\def\@makecaption#1#2{%
 \vskip\abovecaptionskip
  \sbox\@tempboxa{{#1.}\quad #2}%
   \ifdim \wd\@tempboxa >\hsize
    { #1.}\quad #2\par
     \else
  \global \@minipagefalse
   \hb@xt@\hsize{\hfil\box\@tempboxa\hfil}%
   \fi
   \vskip\belowcaptionskip}
\makeatother

\usepackage{latexsym, bm}
\title{\vspace{0cm}{On the anti-Kelul\'{e} problem of cubic graphs}\footnote{This
 work is supported by NSFC (grant nos. 11401279 and 11371180), the Specialized Research Fund
 for the Doctoral Program of Higher Education (No. 20130211120008), the Fundamental Research Funds for the Central Universities (no. lzujbky-2016-102), and General Research Fund of Hong Kong.}}

\date{}

\begin{document}

\author{Qiuli Li$^{a}$,  Wai Chee Shiu$^{b}$\footnote{The
corresponding author.}, Pak Kiu Sun$^{b}$, and Dong Ye$^{c}$ \\
\small{$^{a}$School of Mathematics and Statistics, Lanzhou
University, Lanzhou,}\\ \small{Gansu 730000, China}\\
\small{E-mail address:  qlli@lzu.edu.cn}\\
\small{$^{b}$Department of Mathematics, Hong Kong Baptist University,}\\
\small{Kowloon Tong, Hong Kong,
China} \\
\small{E-mail addresses: wcshiu@hkbu.edu.hk, lionel@hkbu.edu.hk}\\
\small{$^{c}$Department of Mathematical Sciences, Middle Tennessee State University,}\\
\small{Murfreesboro, TN 37132, United States} \\
\small{E-mail address: dong.ye@mtsu.edu} }

\maketitle

\begin{abstract}
 An edge set $S$ of a connected graph $G$ is  called an {\em anti-Kekul\'e set} if $G-S$ is connected and has no perfect matchings, where $G-S$ denotes the subgraph obtained by deleting all edges in $S$ from $G$. The {\em anti-Kekul\'e number} of a graph $G$, denoted by $ak(G)$, is the cardinality of a smallest anti-Kekul\'e set of $G$.  It is NP-complete to find the smallest anti-Kekul\'e set of a graph. In this paper, we
show that the anti-Kekul\'{e} number of  a 2-connected
cubic graph is either 3 or 4, and the anti-Kekul\'{e} number
of a connected cubic bipartite graph is always equal to 4.
Furthermore, a polynomial time algorithm is given to find all smallest anti-Kekul\'{e} sets of a connected cubic graph.

\def\MSC{\par\leavevmode\hbox{\em AMS 2010 MSC:\ }}%
\MSC 05C10, 05C70, 05C90
\end{abstract}

\section{Introduction}

 Let $G$ be a graph. A {\em perfect matching} of a graph $G$ is a set of non-adjacent edges that covers all vertices of $G$. A perfect matching of a graph is also called a {\em Kekul\'e structure} in mathematical chemistry and statistical physics. An edge set $S$ of a connected graph $G$ is called  an {\em anti-Kekul\'e set} if $G-S$ is connected and has no perfect matchings, where $G-S$ denotes the subgraph obtained by deleting all edges in $S$ from $G$. The {\em anti-Kekul\'e number} of a graph $G$, denoted by $ak(G)$, is the cardinality of a smallest anti-Kekul\'e set of $G$. The anti-Kekul\'e number of a graph is hard to determine in general. It has been proved recently that it is NP-complete to determine the anti-Kekul\'e number of a connected bipartite graph by L\"u, Li and Zhang~\cite{Lv16}.




In chemistry and physics, graphs are used to represent the skeletons of molecules, and Kekul\'e structures (or perfect matchings) are used to model special structures of bonds between atoms. For example, for a benzenoid hydrocarbons, graphens or fullerenes, a Kekul\'e structure of these molecules stands for double bonds between atoms. An anti-Kekul\'e set is a set of double bonds whose removal significantly affects the whole molecule structure by the valence bond (VB) theory (cf. \cite{DKlein}). 

A fullerene is a 3-connected plane cubic graph such that every face is either a hexagon or a pentagon. For example, $\text{C}_{60}$ is a fullerene with 60 vertices such that all pentagons are disjoint. The anti-Kekul\'e number of $\mbox{C}_{60}$ is proved to be 4 by
Vuki\v{c}evi\'{c} 
\cite{vuki071}. A leapfrog fullerene is a fullerene obtained by leapfrog operation. Kutnar et al.  \cite{Kutnar09} obtained a bound for the anti-Kekul\'{e} number
of  leapfrog fullerenes as follows.

\begin{thm}[\cite{Kutnar09}]\label{thm:leap}
Let $G$ be a leapfrog fullerene. Then $3\leq ak(G) \leq 4$.
\end{thm}

The above result was improved by Yang et
al. \cite{Yang12} by proving that all fullerenes have anti-Kekul\'{e} number 4.
The anti-Kekul\'{e} numbers of other interesting graphs,
such as benzenoid hydrocarbons \cite{cai13, vuki08},
fence graphs \cite{tang11}, infinite triangular, rectangular
and hexagonal grids \cite{veljan08} as well as cata-condensed phenylenes \cite{Bian11},
 have been investigated.

The result on fullerenes has been generalized to general cubic graphs with
high cyclic edge-connectivity in \cite{Ye}. A graph $G$ is cyclically $k$-edge-connected if $G$ cannot be separated into two
components, each containing a cycle, by deletion of fewer than $k$ edges.
The \emph{cyclic edge-connectivity} $c\lambda(G)$ of a graph
$G$ is the maximum $k$ such that  $G$ is cyclically $k$-edge-connected.
An edge set $S$ is called an  \emph{odd cycle edge-transversal} of a cubic graph $G$ if $G-S$ is bipartite. The size of a smallest odd cycle edge-transversal of $G$ is denoted by $\tau_{\text{odd}}(G)$.



\begin{thm}[\cite{Ye}]\label{ye13}
Let $G$ be a cyclically $4$-edge-connected cubic graph.  Then either $ak(G) = 4$ or $1 \leq \tau_{odd}(G) \leq 3$.
\end{thm}

 The result above can be used to determine the anti-Kekul\'{e} number of fullerenes. Since the smallest odd cycle-transversal of a fullerene graph contains  at least 6 edges and the cyclic edge-connectivity
of a fullerene graph is 5 (see \cite{Tom, qi08}), Theorem~\ref{ye13} implies that every fullerene has anti-Kekul\'{e} number 4. However, Theorem~\ref{ye13} is not applicable to determine the anti-Kekul\'e numbers of some interesting graphs, such
as, 
 some boron-nitrogen fullerenes with low cyclic edge-connectivity, (3,6)-fullerenes etc.

A $(k, 6)$-cage ($k\ge 3$) is a 3-connected cubic planar graph
whose faces are either $k$-gons or hexagons. Do\v{s}li\'{c}  \cite{Tom} shows that $(k, 6)$-cages only exist for $k = 3$, 4 and 5.
A fullerene is a $(5,6)$-cage and the (4,6)-cages and $(3,6)$-cages are usually called (4,6)-fullerenes (or boron-nitrogen fullerens) and (3,6)-fullerenes, respectively.
Many researches have been done to investigate the properties of these graphs in both mathematics and chemistry, such as hamilitionian \cite{Goodey77, Goodey75}, resonance \cite{Rui12, ye09, zhang10},  the forcing matching number \cite{jiang11, zhang101}, and
energy spectra of (3,6)-fullerenes \cite{devos09, john09} which determines their electronic and magnetic
properties \cite{ceul02, szopa03}.

The  cyclic edge-connectivity of $(k,6)$-cages has been obtained by Do\v{s}li\'{c} in
\cite{Tom}. Let $\mathcal{T}$ be a family of (4,6)-fullerenes which
consists of $n$ concentric layers of hexagons (i.e. each layer is a cyclic chain of three hexagons) and is capped on
each end by a cap formed by three quadrangles.


\begin{thm}[\cite{Tom}]\label{cage}
Let $G$ be a $(k, 6)$-cage. Then $c\lambda(G) = 3$ if $G \in  \mathcal{T}$, and  $c\lambda(G) = k$ otherwise.
\end{thm}

In this paper, we consider the anti-Kekul\'{e} number of connected cubic graphs including those with low cyclic
edge-connectivity. 
The following is our first major result. 


\begin{thm}\label{cubic}
If $G$ is a $2$-connected cubic graph, then $3\leq ak(G) \leq 4$.
\end{thm}

Since a leapfrog fullerene is 3-connected, Theorem~\ref{thm:leap} is a direct corollary of Theorem~\ref{cubic}. For bipartite cubic graphs, the  result can be strengthened as follows.

\begin{thm}\label{cubicbipartite}
If $G$ is a connected cubic bipartite graph, then $ak(G)=4$.
\end{thm}

Theorems~\ref{cubic} and \ref{cubicbipartite} can be applied to determine the anti-Kekul\'{e} numbers of boron-nitrogen fullerenes, (3,6)-fullerenes, toroidal and bipartite Klein-bottle fullerenes (see Section 3 for details). 
Based on 
Theorems \ref{cubic} and \ref{cubicbipartite}, a polynomial time algorithm is given to find all smallest anti-Kekul\'{e} sets of a connected cubic graph $G$ in Section 4.

\section{Proofs of main results}

The well-known theorem of Tutte is essential to our proof of the main results.

\begin{thm}[Tutte's Theorem \cite{Tutte47}]\label{tutte} A graph $G$ has a perfect matching if and only if $c_{o}(G-U) \leq |U|$ for any $U \subseteq V(G)$, where
$c_{o}(G-U)$ is the number of odd components of $G-U$.
\end{thm}

By Petersen's Theorem~\cite{petersen}, every cubic graph without bridges has a perfect matching. Therefore, the anti-Kekul\'{e} number of a 2-connected cubic graph $G$ (note that a cubic graph possesses the same connectivity and edge-connectivity) is at least one, that is,  $ak(G)\geq 1$. Indeed, this lower bound can be improved and we present the proof by using Tutte's Theorem. For $X\subseteq V(G)$, let $\partial(X)$ denote the
set of edges with one end in $X$ and the other
end in $V(G)-X$. We also denote $d(X)=|\partial(X)|$.\medskip

\noindent{\emph{Proof of Theorem~\ref{cubic}}.}
Let $A$ be an anti-Kekul\'{e} set of size $ak(G)$. According to the definition, $G':=G-A$ has no perfect matchings. Hence, Theorem~\ref{tutte} implies that there exists $S\subseteq V(G')$ such that $c_{0}(G'-S)>|S|$.  Choose such an $S$ with the maximum size.

\noindent{\bf Claim 1. } {\it $G'-S$ has no even components and it has exactly $|S|+2$ odd components.}

 Suppose by the contrary, $G'-S$ has an even component $H$. Thus, for any given vertex $v\in V(H)$, $H-\{v\}$ has at least one odd component. Let $S'=S\cup \{v\}$, hence $G'-S'$ has at least $c_{0}(G'-S)+1$ odd components. That is, $c_{0}(G'-S')\geq c_{0}(G'-S)+1>|S|+1=|S'|$, contradicting the choice of $S$. Therefore $G'-S$ has no even component.

 Since $G$ is a cubic graph, it has an even number of vertices. This implies that $c_{0}(G'-S)|$ and $|S|$ are of the same parity,  thus $c_{0}(G'-S)|\geq |S|+2$.  For any edge $e\in A$, since $A$ is an anti-Kekul\'{e} set with the smallest cardinality, $G'+e$ has a perfect matching (note that $G'+e$ is the graph with vertex set $V(G')$ and edge set $E(G')\cup \{e\}$). Hence $c_{o}(G'+e-S)\leq |S|$ by Theorem~\ref{tutte}. Moreover, adding any edge $e$ to $G'-S$ will connects at most two odd components.
  Therefore, $|S|\geq c_{o}(G'+e-S)\geq |c_{0}(G'-S)|-2\geq |S|$ and thus, $c_{0}(G'-S)|-2=|S|$.

\noindent
{\bf Claim 2. } {\it Let $G_{i}$, where $1\leq i \leq |S|+2$, be  the odd components of $G'-S$. We have}
\begin{equation}\label{inequality}
\sum_{i=1}^{|S|+2}d(G_{i})-2ak(G)\leq 3|S|.
\end{equation}



We count the number of edges between $S$ and the odd components, which is denoted by $N$, in two different ways. On one hand, $S$ contributes at most $3|S|$ to $N$. On the other hand, all the odd components send out $\sum_{i=1}^{|S|+2}d(G_{i})-2ak(G)$ edges to $N$. Thus $\sum_{i=1}^{|S|+2}d(G_{i})-2ak(G)= N \leq 3|S|$ and the claim holds.

 Since $G$ is 2-edge-connected, $d(G_{i})\geq 2$ for every $i$. By a simple computation $d(G_{i})=3|V(G_{i})|-2|E(G_{i})|$, which implies that $d(G_{i})$ and $|V(G_{i})|$ are of the same parity. Since every $G_{i}$ is an odd component, $|V(G_{i})|$ is odd and hence $d(G_{i})$ is odd, therefore  $d(G_{i})\geq 3$.

Substituting this inequality into Equation (\ref{inequality}), we have $$3(|S|+2)-2ak(G)\leq \sum_{i=1}^{|S|+2}d(G_{i})-2ak(G)\leq 3|S|,$$ and so $ak(G)\geq 3$.

Now we are going to establish an upper bound on $ak(G)$ and it is sufficient to find an anti-Kekul\'{e} set of size 4. Let $a\in V(G)$ and let $b$ as well as $c$ be its two distinct neighbors. Denote the two edges incident with $b$ other than $ab$ by $e_{1}$ and $e_{2}$, similarly, denote the two edges incident with $c$ other than $ac$ by $e_{3}$ and $e_{4}$.  Hence, removing $E_{a}=\{e_{1}, e_{2}, e_{3}, e_{4}\}$ from $G$ will obtain a subgraph without perfect matchings. Therefore, $E_{a}$ is an anti-Kekul\'{e} set if  $G-E_{a}$ is connected  (there exists some vertex $a$ such that $G-E_{a}$ is not connected, see Fig.~\ref{notconnected}). Consider the following two cases according to the different connectivities.

\begin{figure}[H]
\begin{center}
\includegraphics[totalheight=4cm]{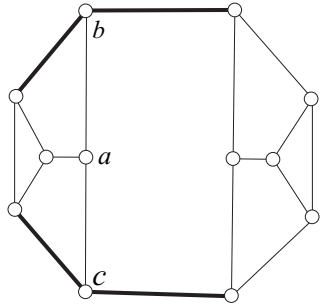}
\caption{\label{notconnected}\small{$E_{a}$ are the bold edges.}}
\end{center}
\end{figure}

\noindent
{\bf Case 1. } {\it $G$ is $3$-connected.}

We are going to prove that for any vertex $a$ in $G$, $G-E_{a}$ is connected. Suppose by the contrary that $G-E_{a}$ is not connected. Then the vertices of $G$ are divided into two parts $X$ and $\overline{X}$ with a subset $E'\subseteq E_{a}$ connecting them. Since $G$ is 3-connected, $E'$ consists of three or four edges in $E_{a}$. If it contains three edges, by symmetry, we assume $E'=\{e_{1}, e_{2}, e_{3}\}$, then $\{ab, e_{3}\}$ is a 2-edge-cut and a contradiction occurs (note that a 2-edge-cut is an edge-cut of size 2). If $E'$ contains four edges, that is $E'=E_{a}=\{e_{1}, e_{2}, e_{3}, e_{4}\}$, then $\{ab, ac\}$ forms a 2-edge-cut, which is a contradiction.

As a result, $G-E_{a}$ is connected and $G-E_{a}$ has no perfect matchings. Therefore, $E$ is an anti-Kekul\'{e} set of size 4  and we have $ak(G)\leq 4$.

\noindent
{\bf Case 2. }{\it $G$ has connectivity $2$.}

Since $G$ is 2-connected but not 3-connected, there exist 2-edge-cuts. Moreover, each 2-edge-cut is an independent set, otherwise the third edge adjacent to them is a bridge, which contradicts that $G$ is 2-connected. Every 2-edge-cut will split $G$ into exactly two subgraphs. Among those subgraphs, denote the one with the smallest cardinality by $G'$ and  the corresponding 2-edge-cut by $E=\{e_{4}, e_{5}\}$. Also, denote the end-vertices of $e_{4}$ and $e_{5}$ in $G'$ by $v$ and $u$, respectively. Moreover, let $G''$ be the other subgraph obtained by deleting $E$.

\noindent{\bf Claim 3.} {\it $uv\notin E(G)$.}

Assume $uv \in E(G)$. Then the edges incident with $u$ or $v$ other than $uv$, $e_{4}$ and $e_{5}$ form a 2-edge-cut. The deletion of this 2-edge-cut creates a subgraph with  cardinality smaller than $G'$, contradicting the choice of $E$.

\noindent{\bf Claim 4.} {\it No $2$-edge-cut of $G$ contains an edge $e\in E(G')$.}

Suppose the claim is false and there exists  $e \in E(G')$ that lies in some 2-edge-cut $E'=\{e, e'\}$. No matter where $e'$ lies in, the subgraph induced by $V(G'')\cup \{u\}$ or $V(G'')\cup \{v\}$ belongs to a component created by the deletion of $E'$. Thus the cardinality of the other component is smaller  than $G'$, which contradicts  the choice of $E$ and  the claim holds.

Let $s$ be a neighbor of $v$ in $G'$. Since $s$ is of degree 3, there exists a neighbor $t$ $(\neq u)$ of it in $G'$. Let $e_{1}$ and $e_{2}$ be two incident edges of $t$ other than $st$, and let $e_{3}$ be the edge incident with $v$ other than $sv$ and $e_{2}$. We claim that $\{e_{1}, e_{2}, e_{3}, e_{4}\}$ is an anti-Kekul\'{e} set. It is obvious that $G-\{e_{1}, e_{2}, e_{3}, e_{4}\}$ has no perfect matchings.  If $G-\{e_{1}, e_{2}, e_{3}, e_{4}\}$ is not connected, then,  similar to Case 1, we obtain a 2-edge-cut containing at least one edge in $G'$. This is a contradiction and completes the proof.

\vspace{.3cm}

The condition ``2-edge-connected'' in Theorem~\ref{cubic} is necessary because there exists cubic graphs with bridges and their anti-Kekul\'{e} number is less than 3 (see Fig.~\ref{012}). More precisely, we have the following result.

\begin{thm}\label{cubic-bridge}
If $G$ is a connected cubic graph with bridges, then $ak(G)\leq 2$.
\end{thm}

\begin{proof}
Choose a bridge such that the deletion of it will give a subgraph with the smallest cardinality, we denote this subgraph by $G'$ and the corresponding bridge by $e$.  Let the end vertex of $e$ in $G'$ be $u$ and let $v$ be a neighbor of $u$ in $G'$. Moreover, let the two other edges incident with $v$ other than $uv$ be $e_{1}$ and $e_{2}$. Similar to the proof of Case 2 in Theorem~\ref{cubic}, we have $G-\{e_{1}, e_{2}\}$ is connected. Since any bridge separates $G$ into two odd components, any perfect matching $M$ of $G$ should contain $e$. Also, $M$ contains one edge in $\{e_{1}, e_{2}\}$ and thus, $G-\{e_{1}, e_{2}\}$ has no perfect matchings. As a result,  $\{e_{1}, e_{2}\}$ is an anti-Kekul\'{e} set and so $ak(G)\leq 2$.
\end{proof}

Fig.~\ref{012} presents three cubic graphs with anti-Kekul\'{e} numbers 0, 1 and 2, and the sets of bold edges denote the smallest anti-Kekul\'{e} sets respectively.

\begin{figure}[H]
\begin{center}
\includegraphics[totalheight=7cm]{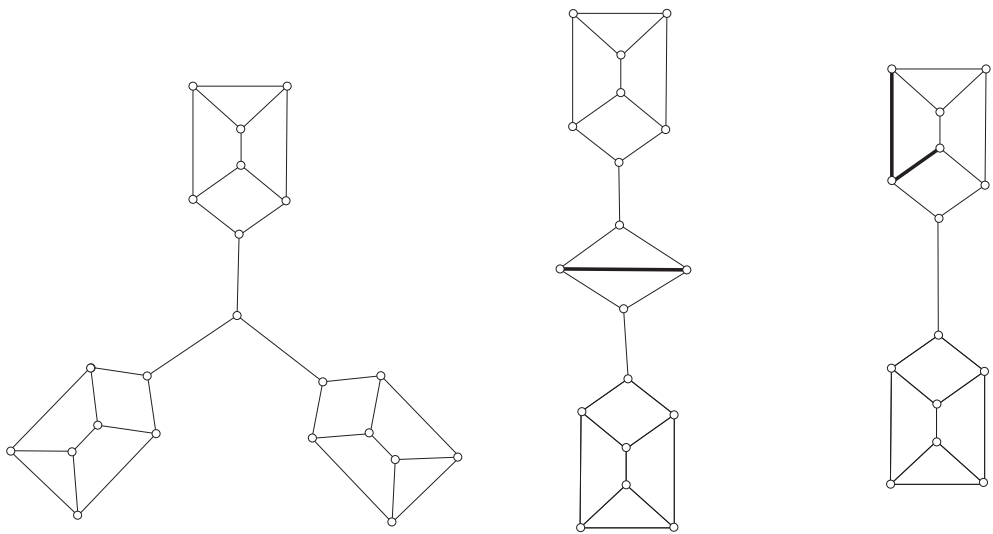}
\caption{\label{012}\small{Three cubic graphs with anti-Kekul\'{e} numbers 0, 1 and 2 respectively.}}
\end{center}
\end{figure}

If the graphs being considered are bipartite, then a stronger result can be obtained by using Hall's Theorem.


\begin{thm}[Hall's Theorem~\cite{Hall35}] \label{hall}
Let $G$ be a bipartite graph with bipartition $W$ and $B$. Then $G$ has a perfect matching if and only if $|W|=|B|$ and for any $U\subseteq W$, $|N(U)|\geq |U|$ holds.
\end{thm}\medskip


\noindent{\emph{Proof of Theorem~\ref{cubicbipartite}}.}
First we show that a connected cubic bipartite graph is essentially 2-connected. By Theorem~\ref{hall}, a $k$-regular bipartite graph contains a perfect matching. Removing that perfect matching will result in a $(k-1)$-regular bipartite graph, and the same argument can be applied repeatedly. Finally, we deduce that a $k$-regular bipartite graph can be decomposed into $k$ disjoint perfect matchings. 
 Since $G$ is a cubic bipartite graph, it can be decomposed into three disjoint perfect matchings $M_{1}$, $M_{2}$ and $M_{3}$, that is, $E(G)=M_{1}\cup M_{2}\cup M_{3}$. For any $e\in E(G)$, without loss of generality, let $e\in M_{1}$. Since $M_{1}$ and $M_{2}$ are disjoint perfect matchings of $G$, $M_{1}\cup M_{2}$ consists of disjoint even cycles and $e$ lies in one of them. Hence $e$ is not a bridge and $G$ is 2-edge-connected. Furthermore, a 2-edge-connected cubic graph is 2-connected, thus $G$ is 2-connected. 

According to Theorem~\ref{cubic}, we have $3\leq ak(G)\leq 4$.
Suppose by the contrary that $ak(G)\neq4$, that is, $ak(G)=3$. Let $A=\{e_{1}, e_{2}, e_{3}\}$ be an anti-Kekul\'{e} set.  Then $G-A$ has no perfect matchings. Assume $W$ and $B$ are the bipartition of $G$. According to Hall's theorem, there exists $S\subseteq W$ such that
\begin{equation}\label{a}
|N_{G-A}(S)|\leq |S|-1.
\end{equation}

On the other hand, since $A$ is an anti-Kekul\'{e} set with the smallest cardinality, we have
\begin{equation}\label{b}
|S|\leq|N_{G-A+e_{i}}(S)|
 \end{equation}
 for $i =1, 2$ and $3$.
Adding an edge $e_{i}$ to $G-A$ will increase the neighbors of $S$ by one (at most). Hence
 \begin{equation}\label{c}
 |N_{G-A+e_{i}}(S)|\leq |N_{G-A}(S)|+1.
  \end{equation}
 Combining inequalities (\ref{a}), (\ref{b}) and (\ref{c}), we obtain   $|S|=|N_{G-A}(S)|+1$. Let $S'=N_{G-A}(S)$.  The edges going out from $S$ are divided into two parts: either goes into $A$ or goes into $S'$. Thus the number of edges between $S$ and $S'$ are $3|S|-3$. Since $|S'|=|S|-1$, there is no edge between $S'$ and $W-S$. Therefore, $A$ is an edge-cut, which contradicts  the definition of anti-Kekul\'{e} set.

\section{Applications}


In this section, we apply Theorems~\ref{cubic} and \ref{cubicbipartite} to obtain the anti-Kekul\'{e} numbers of several
families of interesting graphs, such as boron-nitrogen fullerenes and $(3,6)$-fullerenes.

\begin{thm}
If $G$ is a $(4,6)$-fullerene, then $ak(G)=4$.
\end{thm}
\begin{proof}
Since $G$ is bipartite, the result follows immediately by Theorem~\ref{cubicbipartite}.
\end{proof}

Note that there are two classes of boron-nitrogen fullerenes, one with cyclic edge-connectivity 3 and the other with cyclic edge-connectivity 4. The anti-Kekul\'{e} number of the latter  can  be obtained by Theorem~\ref{ye13}. Now we are going to determine the anti-Kekul\'{e} number of (3,6)-fullerenes and the following lemma is required. A cyclic 3-edge-cut of a (3,6)-fullerene is called trivial if  it is formed by the edges incident to a triangle in common. A 3-edge-cut is called trivial if they are incident to a common vertex. Let $T_{n}$ $(n \geq 1)$ be the graph consisting of $n$ concentric layers of hexagons, capped on
each end by a cap formed by two adjacent triangles (see Fig.~\ref{3,6}).

\begin{Lemma}[\cite{Rui12}]\label{trivial}
(i) Every cyclic $3$-edge-cut of a $(3,6)$-fullerene with connectivity $3$ is trivial; (ii) The connectivity of a $(3,6)$-fullerene is $2$ if and only if it is isomorphic to $T_{n}$ for some $n\geq 1$.
\end{Lemma}

\begin{thm}
If $G$ is a $(3,6)$-fullerene, then $ak(G)=3$.
\end{thm}

\begin{proof}
Let $G$ be a $(3,6)$-fullerene. Note that a $(3,6)$-fullerene has connectivity either 2 or 3. Thus, Theorem~\ref{cubic} implies that $3\leq ak(G)\leq 4$. To show that $ak(G)=3$, it suffices to give an anti-Kekul\'{e} set of size 3.


 First, assume that the connectivity of $G$ is 2. By 
Lemma~\ref{trivial}, $G$ has two triangles sharing a common edge. Let $S$ be the edge set of such a triangle (see  Fig.~\ref{3,6}). Then $G-S$ has two vertices of degree 1 adjacent to a common vertex. Hence $G-S$ has no perfect matching. Clearly, $G-S$ is connected. So $S$ is an anti-Kekul\'e set of size 3.

\begin{figure}[H]
\begin{center}
\includegraphics[totalheight=5.5cm]{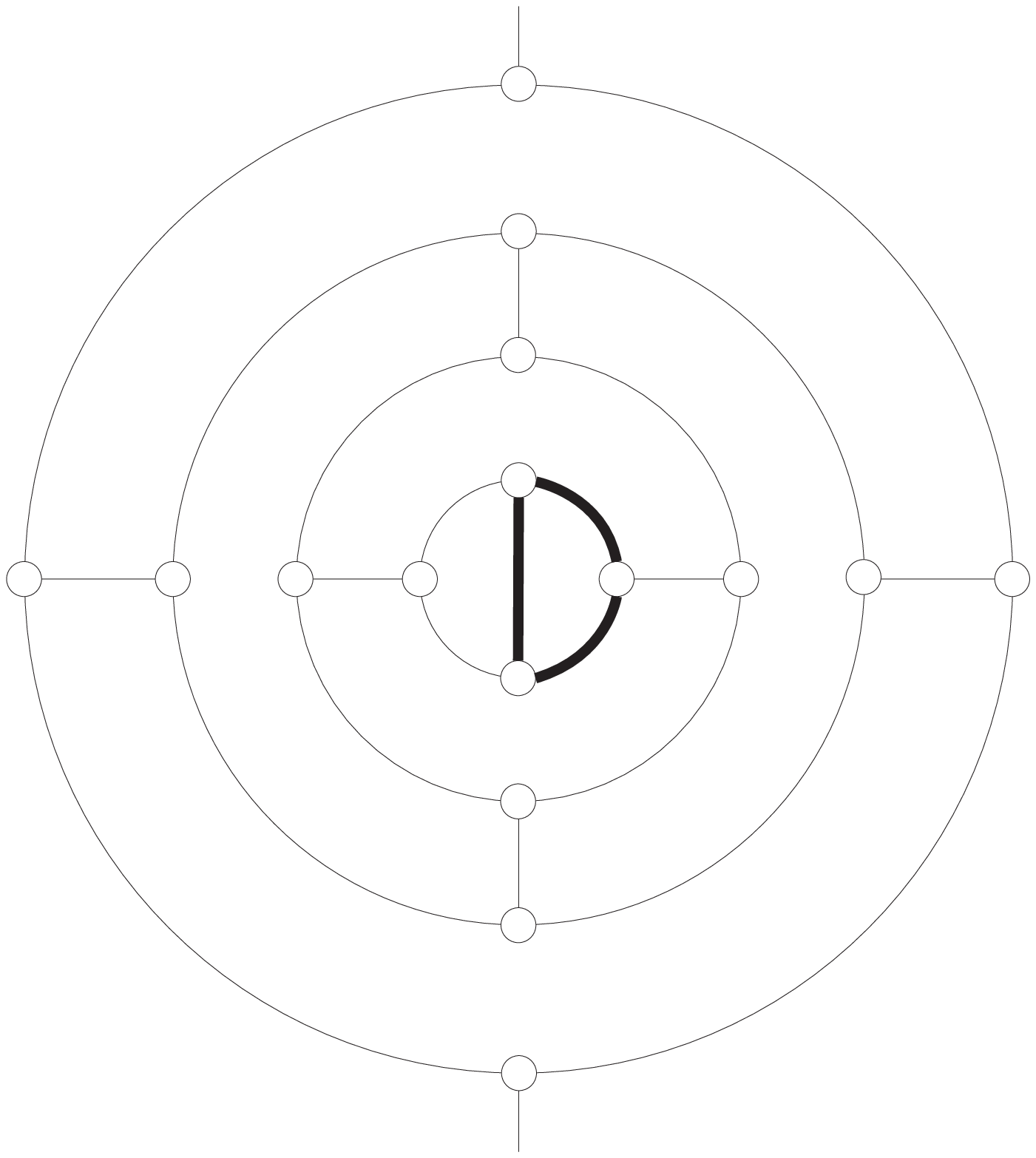}
\caption{\label{3,6}\small{A (3,6)-fullerene $T_{3}$ and the set of bold edges form an anti-Kekul\'{e} set of it.}}
\end{center}
\end{figure}

In the following, assume that $G$ is 3-connected. 
Let $S$ be a 3-edge-cut and let $G_{1}$ as well as $G_{2}$ be the two components of $G-S$. If $S$ is not trivial, then  $d(G_{i})=3$ and $|V(G_{i})|\geq 2$ for $i=1,2$. Since $|V(G_{i})|$ and $d(G_{i})$ are of the same parity, it follows that $|V(G_{i})|\geq 3$. Hence \[|E(G_{i})|-|V(G_{i})|=\frac{3|V(G_{i})|-3}{2}-|V(G_{i})|=\frac{|V(G_{i})|-3}{2}\geq 0.\] Therefore both $G_{1}$ and $G_2$ contain cycles. So $S$ is a cyclic 3-edge-cut. By Lemma \ref{trivial}, $S$ is a trivial cyclic 3-edge-cut.
Hence, a 3-edge-cut of $G$ is either a trivial 3-edge-cut or a trivial cyclic 3-edge-cut.

Let $abc$ be a triangle of $G$. Let $e_1$ be the edge incident with $a$ but not inside of the triangle, and $e_2$ be the edge  incident with $c$ but not contained in the triangle. The edge set $S=\{e_{1}, e_{2}, ac\}$ does not isolate a vertex or a triangle. So 
$S$ is not an edge-cut.  In the subgraph $G-S$, both $a$ and $c$ have degree 1 and both of them are adjacent to $b$. So $G-S$ has no perfect matching. Therefore, $S$ is an anti-Kekul\'{e} set. This completes the proof.
\end{proof}

Furthermore, since a toroidal fullerene or a bipartite Klein-bottle fullerene, whose definitions can be found in \cite{Dong09}, is a cubic bipartite graph, the following result is a direct consequence of Theorem~\ref{cubicbipartite}. Note that this result can also be deduced from Theorem~\ref{ye13}.

\begin{cor}
If  $G$ is either a toroidal fullerene or a bipartite Klein-bottle fullerene, then $ak(G)=4$.
\end{cor}


\section{Finding all the smallest anti-Kekul\'{e} sets} 

The anti-Kekul\'e problem of graphs can be stated as follows.\medskip

\noindent{\bf Instance: } {\it A nonempty graph $G = (V,E)$ having a perfect matching and a positive $k$.}

\noindent{\bf Question: }{\it Dose there exist a subset $B\subseteq E$ with $|B|\leq k$ such that $G'= (V,E\setminus B)$ is connected and $G'$ has no Kekul\'e structure?}
\medskip

In \cite{Lv16}, the authors showed that  anti-Kekul\'e problem on bipartite graphs is NP-complete. 
So it is hard to find a smallest anti-Kekul\'e set of a given graph. However, for cubic graphs, the problem becomes much easier by Theorems~\ref{cubic} and \ref{cubicbipartite}: all the smallest anti-Kekul\'{e} sets of a cubic graph can be found in polynomial time. The algorithm finding all smallest anti-Kekul\'e sets $S$ of a cubic graph $G$ depends on how to find a maximum matching in the graph $G-S$. If the maximum matching of $G-S$ has size exactly $n/2$ where $n$ is the number of vertices of $G$, then it is a perfect matching of $G-S$.


%

For a given graph $G$ with $n$ vertices, Edmonds \cite{Edmonds} found an algorithm to find a maximum matching of $G$ in $O(n^4)$ steps, which is the  blossom algorithm. An efficient implementation of Edmonds' algorithm takes $O(n^3)$ steps to find a maximum matching  \cite{H. N. Gabow}. Later, Micali and Vazirani \cite{MV, V}, Gabow and Tarjan \cite{GT}, and Blum \cite{B} have given algorithms to find a maximum matching of $G$ in $O(\sqrt{n} m)$ steps, where $m$ is the number of edges of $G$.

\begin{thm}[\cite{MV, GT, B}]
Let $G$ be a graph with $n$ vertices and $m$ edges. It takes $O(\sqrt{n} m)$ steps to find a maximum matching of $G$. 
\end{thm}

The connectedness of a graph $G$ with $n$ vertices can be determined by the breadth-first search (BFS) algorithm, which takes $O(n)$ steps. Based on Theorems~\ref{cubic} and \ref{cubicbipartite}, by applying the BFS algorithm and the maximum matching algorithm to $G-S$, we can find all smallest anti-Kekul\'e sets $S$ of a cubic graph $G$. 

\medskip

 \noindent{\bf Algorithm} (Finding all smallest anti-Kekul\'e sets)

\noindent{\bf Input: }{\it A cubic graph $G$ with $n$ vertices;}

\noindent{\bf Output: } {\it All the smallest anti-kekul\'e sets of $G$.} 







\begin{enumerate}[Step 1.]
\item Let $k=0$. Use the maximum matching algorithm on $G$. If $G$ has a maximum matching of size $n/2$, go to Step~2. Otherwise, $ak(G)=0$ and stop.  Then $\varnothing$ is the only smallest anti-Kekul\'e set of $G$.

\item Set $k\leftarrow k+1$. Screen all edge subsets $S$ of size $k$ and  let $\mathcal F_k:=\{S \, | |S|=k\mbox { and } S\subset E(G)\}$. Go to Step~3.

\item Choose an $S$ from $\mathcal F_k$, apply the BFS algorithm to find a spanning tree of $G-S$. If $G-S$ has no spanning tree, go to Step~4. Otherwise, apply the maximum matching algorithm to $G-S$. If $G-S$ has a maximum matching of size $n/2$, go to Step~4. Otherwise, label $S$ as a smallest anti-Kekul\'e set and go to Step~4. 

\item Set $\mathcal F_k\leftarrow \mathcal F_k\backslash \{S\}$. If $\mathcal F_k\ne \emptyset$, return to Setp~3. Otherwise, go to Step~5.

\item If there is no labelled edge set, go to Step~2. Otherwise, output all labelled sets and stop.
\end{enumerate}

 The screening process in Step 2 takes at most $m \choose k$ steps. By Theorem~\ref{cubic}, $k\le 4$. So the worst case takes $m\choose 4$ steps, which is $O(m^4)$ steps. It takes $O(n)$ steps to run BFS algorithm for $G-S$ and $O(\sqrt n m)$ steps to find a maximum matching of $G-S$. So for a given $S$, it takes at most $O(\sqrt n m)$ steps to determine whether it is an anti-Kekul\'e set or not. Therefore, the worest case takes $O(\sqrt n m^5)$ steps to find all smallest anti-Kekul\'e sets of $G$. 
Since  $G$ is a cubic graph, $m=3n/2$. So we have the following result.

\begin{thm}
 Let $G$ be a connected cubic graph with $n$ vertices. Then it takes   $O(n^{11/2})$ steps to find out all the smallest anti-Kekul\'e sets of $G$.
\end{thm}

\end{document}